\documentclass[12pt]{article}
\usepackage{fancyhdr}
\usepackage{amsmath}
\usepackage{amsfonts}
\usepackage{setspace}
\usepackage{enumerate}
\usepackage{epsfig}
\usepackage{subfigure}
\usepackage{rotating}
\usepackage{color}
\usepackage{multirow}

\numberwithin{equation}{section}
\begin{document}
\baselineskip=22pt
\bibliographystyle{plain}

\newcommand{\ve}[1]{\mbox{\boldmath$#1$}}
\newcommand{\be}{\begin{equation}}
\newcommand{\ee}{\end{equation}}
\newcommand{\ben}{\begin{equation*}}
\newcommand{\een}{\end{equation*}}
\newcommand{\bc}{\begin{center}}
\newcommand{\ec}{\end{center}}
\newcommand{\bal}{\begin{align*}}
\newcommand{\enal}{\end{align*}}
\newcommand{\al}{\alpha}
\newcommand{\bt}{\beta}
\newcommand{\gm}{\gamma}
\newcommand{\de}{\delta}
\newcommand{\la}{\lambda}
\newcommand{\om}{\omega}
\newcommand{\Om}{\Omega}
\newcommand{\Gm}{\Gamma}
\newcommand{\De}{\Delta}
\newcommand{\Th}{\Theta}
\newcommand{\nno}{\nonumber}
\newtheorem{theorem}{Theorem}[section]
\newtheorem{lemma}{Lemma}[section]
\newtheorem{assum}{Assumption}[section]
\newtheorem{claim}{Claim}[section]
\newtheorem{proposition}{Proposition}[section]
\newtheorem{corollary}{Corollary}[section]
\newtheorem{definition}{Definition}[section]
\newtheorem{remark}{Remark}[section]
\newenvironment{proof}[1][Proof]{\begin{trivlist}
\item[\hskip \labelsep {\bfseries #1}]}{\end{trivlist}}
\newenvironment{proofclaim}[1][Proof of Claim]{\begin{trivlist}
\item[\hskip \labelsep {\bfseries #1}]}{\end{trivlist}}

\def \qed {\hfill \vrule height7pt width 5pt depth 0pt}
\def\refhg{\hangindent=10pt\hangafter=1}
\def\refmark{\par\vskip 2.50mm\noindent\refhg}

\title{\textbf{ Stability of Coalescence Hidden variable Fractal Interpolation Surfaces }}
\date{}
\author{G. P. Kapoor$^1$  and  Srijanani Anurag Prasad$^2$}
\maketitle \bc Department of Mathematics and Statistics\\
Indian Institute of Technology Kanpur\\  Kanpur 208016, India\\
$^1$ gp@iitk.ac.in  $^2$ jana@iitk.ac.in\ec

\onehalfspacing

\begin{abstract}
In the present paper, the stability of Coalescence Hidden variable Fractal Interpolation Surfaces(CHFIS) is established.
The estimates on error in approximation of the data generating function by CHFIS are found
when there is a perturbation in independent, dependent and hidden variables. It is proved that any small
perturbation in any of the variables of generalized interpolation data results in only small
perturbation of CHFIS.  Our results are likely to be useful in investigations of texture of surfaces
arising from the simulation of surfaces of rocks, sea surfaces, clouds
and similar natural objects wherein the generating function depends on more than one variable.
\end{abstract}


\newpage

\section{Introduction}
The theory of fractal interpolation has become a powerful tool in applied science and engineering since Barnsley~\cite{barnsley86}
introduced Fractal Interpolation Function (FIF) using the theory of Iterated Function System (IFS). Massopust~\cite{massopust90}
extended this concept to Fractal Interpolation Surface (FIS) using IFS wherein he assumed the surface as triangular simplex and
interpolation points on the boundary to be co-planar. In view of lack of flexibility in his construction,
Geronimo and Hardin~\cite{geronimo93} generalized the construction of FIS by allowing more general boundary data. Subsequently,
Xie and Sun~\cite{xie97} used bivariate functions on rectangular grids with arbitrary contraction factors and without any condition
on boundary points to construct Bivariate FIS. Dalla \cite{dalla02} improvised this construction by using collinear boundary points
and proved that the attractor is continuous FIS. However, all the constructions mentioned above lead to self-similar attractors.

A non-diagonal IFS that generates both self-affine and non-self-affine FIS simultaneously depending on the free variables and
constrained variables on a general set of interpolation data is constructed in~\cite{chand03}. The attractor of such an IFS
is called Coalescence Hidden-variable Fractal Interpolation Surface (CHFIS). Since the CHFIS passes through the given data points,
any small perturbation in the data points results in the perturbation of the corresponding CHFIS.

The construction of a Coalescence Hidden-variable Fractal Interpolation Function (CHFIF)  of one variable and
investigation of its stability is studied in~\cite{chand07,chand08}. A CHFIF is an important tool in the study of highly uneven
curves like fractures in rocks, seismic fracture, lightening, ECG, etc. However, it can not be applied for the study of
highly uneven surfaces such as surfaces of rocks~\cite{xie97}, sea surfaces~\cite{osborne86},
clouds~\cite{zimmermann92} and many other naturally occuring objects for which the generating function depends on more than
one variable. A CHFIS is a preferred choice for the study of these naturally occurring objects. The quantification of smoothness of
such surfaces in terms of Lipschitz exponent of its corresponding CHFIS is investigated recently
in~\cite{janani08_1}. The purpose of the present paper is to investigate the stability of such CHFIS.
The estimates on error in approximation of the data generating function by CHFIS are found individually
when there is a perturbation in independent, dependent or hidden variable. These estimates together give the total error estimate on
CHFIS when there is perturbation in all these  variables simultaneously.
It is proved that any small perturbation in any of the variables of generalized interpolation data results in only small
perturbation of CHFIS. Unlike the case of CHFIF, the stability of CHFIS is studied with respect to Manhattan metric
and requires the perturbations in generalized interpolation data to be governed by an invariance of ratio condition.
Our results are likely to find applications in texture of surfaces of naturally occurring objects like surfaces of  rocks,
sea surfaces, clouds, etc..

A brief introduction of CHFIS is given in Section~\ref{sec:two}. In Section~\ref{sec:three}, some auxiliary results that are needed
to establish the stability of CHFIS are derived. Our main stability result is established in Section~\ref{sec:four} via stability
results found individually for perturbation in independent variables, the dependent variable and the hidden variable. Finally,
these results are illustrated in Section~\ref{sec:five} through simulation of a sample surface with given data as well as with
perturbed data obtained by small variations of concerned variables in the given data.

\section{Preliminaries}\label{sec:two}
For the given interpolation data  $ \{(x_i,y_j,z_{i,j}):i=0,1,\ldots,N \ \mbox{and}\ j=0,1,\ldots,M\},$ where $z_{i,j} \in (a,b)$ and
$-\infty < a < b < \infty$, consider the generalized interpolation data
$ \triangle=\{(x_i,y_j,z_{i,j},t_{i,j}) :i=0,1,\ldots,N \ \mbox{and} \  j=0,1,\ldots,M \},$
where $t_{i,j} \in (c,d)$ and $-\infty < c <d < \infty$. Set
$I=[x_0,x_N],\ J=[y_0,y_M],\  S=I \times J ,\ D=(a,b) \times (c,d), \
I_n=[x_{n-1},x_n],\ J_m=[y_{m-1},y_m] \ \mbox{and} \  S_{n,m}= I_n \times J_m$  for
$n=1,\ldots,N \ \mbox{and}\ m=1,\ldots,M.$
Let,  the mappings $ \phi_n :I \rightarrow I_n , \ \psi_m: J \rightarrow J_m $  and $ F_{n,m} : S \times D \rightarrow D\ $
for  $n=1,\ldots,N ,\ \mbox{and}\  m=1,\ldots,M\ $ be defined as follows:
\begin{align*}
\phi_n(x)& = x_{n-1} + \frac{x_n - x_{n-1}}{x_N-x_0}\ (x - x_0)  ,\\
\psi_m(y)& = y_{m-1} + \frac{y_m - y_{m-1}}{y_M-y_0}\  (y - y_0) ,\\
\ F_{n,m}(x,y,z,t) & = (e_{n,m}\  x +f_{n,m}\ y + \al_{n,m}\ z + \bt_{n,m}\ t + g_{n,m}\ x y +  k_{n,m}, \nno \\
& \quad \quad \quad \quad \quad \quad \tilde{e}_{n,m}\ x + \tilde{f}_{n,m}\ y + \gm_{n,m}\ t + \tilde{g}_{n,m}\ x y + \tilde{k}_{n,m}).
\end{align*}
Here, $\ \al_{n,m}$  and $\gm_{n,m}$  are free variables chosen such
that $|\al_{n,m}| < 1$ and  $|\gm_{n,m}| < 1.$ $|\bt_{n,m}|$ is a
constrained variable chosen such that $|\bt_{n,m}| + |\gm_{n,m}| <
1.$ Let the function $F_{n,m}$ satisfy the following join-up
condition:
\begin{align} \label{eq:Fnmcond}
\left.
\begin{array}{rcl}
F_{n,m}(x_0,y_0,z_{0,0},t_{0,0})&=&(z_{n-1,m-1},t_{n-1,m-1})  \\
F_{n,m}(x_N,y_0,z_{N,0},t_{N,0})&=&(z_{n,m-1},t_{n,m-1})  \\
F_{n,m}(x_0,y_M,z_{0,M},t_{0,M})&=&(z_{n-1,m},t_{n-1,m})  \\
F_{n,m}(x_N,y_M,z_{N,M},t_{N,M})&=&(z_{n,m},t_{n,m}).
\end{array} \right\}
\end{align}
Using the condition \eqref{eq:Fnmcond}, the values of $e_{n,m},\
f_{n,m},\ g_{n,m},\ k_{n,m},\ \tilde{e}_{n,m},\ \tilde{f}_{n,m},\
\tilde{g}_{n,m}$ and $\tilde{k}_{n,m}$ are determined as follows:
\begin{align} \label{eq:coef} \left. \begin{array}{rcl}
g_{n,m}&=& \frac{z_{n-1,m-1}-z_{n-1,m}-z_{n,m-1}+z_{n,m} - \al_{n,m} z_{eva} - \bt_{n,m} t_{eva}}{(x_0 - x_N) (y_0 - y_M)}  \\
e_{n,m}&=& \frac{z_{n-1,m-1}-z_{n,m-1}-\al_{n,m}(z_{0,0}-z_{N,0})-\bt_{n,m}(t_{0,0}-t_{N,0})-g_{n,m}(x_0-x_N)y_0}{(x_0-x_N)}  \\
f_{n,m}&=& \frac{z_{n-1,m-1}-z_{n-1,m}-\al_{n,m}(z_{0,0}-z_{0,M})-\bt_{n,m}(t_{0,0}-t_{0,M})-g_{n,m}(y_0-y_M)x_0}{(y_0-y_M)}  \\
k_{n,m}&=&  z_{n,m} - e_{n,m} x_{N} - f_{n,m} y_{M} -\al_{n,m} z_{N,M} - \bt_{n,m} t_{N,M} - g_{n,m} x_{N} y_{M}  \\
\tilde{g}_{n,m}&=& \frac{t_{n-1,m-1}-t_{n-1,m}- t_{n,m-1}+ t_{n,m} - \gm_{n,m}\ t_{eva}}{(x_0 - x_N) (y_0 - y_M)} \\
\tilde{e}_{n,m}&=& \frac{t_{n-1,m-1}-t_{n,m-1}- \gm_{n,m}(t_{0,0}- t_{N,0}) - \tilde{g}_{n,m}(x_0 - x_N)y_0}{(x_0 - x_N)} \\
\tilde{f}_{n,m}&=& \frac{t_{n-1,m-1}-t_{n-1,m}- \gm_{n,m}(t_{0,0}- t_{0,M}) - \tilde{g}_{n,m}(y_0 - y_M)x_0}{(y_0 -y_M)}\\
\tilde{k}_{n,m}&=&  t_{n,m} - \tilde{e}_{n,m} x_N - \tilde{f}_{n,m} y_M - \gm_{n,m} t_{N,M} - \tilde{g}_{n,m} x_N y_M
\end{array} \right\} \end{align}
where, $\ z_{eva}=z_{N,M} - z_{N,0} - z_{0,M} + z_{0,0}\  $  and  $\ t_{eva}=t_{N,M} - t_{N,0} - t_{0,M} +  t_{0,0}.$

Now define the functions $G_{n,m}(x,y,z,t)$  and $\om _{n,m}(x,y,z,t)$ as

\noindent$ G_{n,m}(x,y,z,t) = { \small \left\{
    \begin{array}{llll}
    F_{n+1,m}(x_0,y,z,t), & x=x_N, & n=1,\ldots,N-1, & m=1,\ldots,M \\
    F_{n,m+1}(x,y_0,z,t), & y=y_N, & n=1,\ldots,N    & m=1,\ldots,M-1 \\
    F_{n,m}(x,y,z,t), & \mbox{otherwise}. &  &
    \end{array} \right. }$

and $\om _{n,m}(x,y,z,t)= \left( \phi_n(x), \psi_m(y), G_{n,m}(x,y,z,t)\right)$. Then,
$\{\mathbb{R}^4, \om_{n,m} : n=1,\ldots N;\  m=1, \ldots M \}$ constitutes an IFS for the generalized interpolation data $\triangle$.
It is known \cite{chand03} that there exists a metric $\tau$ on $\mathbb{R}^4,$ equivalent to the Euclidean metric, such that the IFS
is hyperbolic with respect to $\tau $ and there exists a unique non-empty compact set $G \subseteq \mathbb{R}^4$ such that
$ G = \bigcup\limits_{n=1}^N \bigcup\limits_{m=1}^M \omega_{n,m}(G).$ The set $G$ is called the attractor of the IFS for the given
interpolation data. Further, G is the graph of a continuous function $F:S \rightarrow \mathbb{R}^2$
such that $F(x_i,y_j)=(z_{i,j},t_{i,j})$ for $i=0,1,\ldots, N$ and $j=0,1,\ldots,M $ i.e.
$G=\{(x,y,F(x,y)):(x,y)\in S \ \mbox{and}\ F(x,y)=(z(x,y),t(x,y))\}.$
\begin{definition}
Let  $ F(x,y) $ be written component-wise as $F(x,y) =(F_1(x,y),F_2(x,y)).$
The \textbf{Coalescence Hidden-variable Fractal Interpolation Surface (CHFIS) } for the given interpolation data
$\{(x_i,y_j,z_{i,j}):i = 0,1,\ldots,N \ \mbox{and}\ j = 0,1,\ldots,M \}$ is defined as the surface  $z=F_1(x,y)$ in
$\mathbb{R}^3 .$
\end{definition}
It is easily seen that the function $F(x,y)$ described above
satisfies
\begin{align*}
F(x,y)=G_{n,m}(\phi_n^{-1}(x), \psi_m^{-1}(y),F(\phi_n^{-1}(x), \psi_m^{-1}(y)))
\end{align*}
for all $(x,y) \in S_{n,m}, \ n=1,2,\ldots,N \ \mbox{and}\
m=1,2,\ldots,M$. Consequently, if the function~$G_{n,m}(x,y,z,t)$ is
written component-wise as $G_{n,m}(x,y,z,t) = (G_{n,m}^1(x,y),
G_{n,m}^2(x,y))$, then
\begin{align*}{\small
\begin{array}{llll}
&F_1(\phi_n(x),\psi_m(y))= G_{n,m}^1(x,y,F_1(x,y),F_2(x,y))= \al_{n,m}\ F_1(x,y) + \bt_{n,m}\ F_2(x,y)
+  p_{n,m}(x,y)\\
&  F_2(\phi_n(x),\psi_m(y)) = G_{n,m}^2(x,y,F_1(x,y),F_2(x,y))= \gm_{n,m}\ F_2(x,y)+q_{n,m}(x,y)
\end{array}}
\end{align*}
where,
\begin{align}\label{eq:pnmqnm}{\small \left.
\begin{array}{rcl}
p_{n,m}(x,y) = e_{n,m}\ x + f_{n,m}\ y +g_{n,m}\ xy +k_{n,m}  \\
q_{n,m}(x,y)= \tilde{e}_{n,m}\ x +\tilde{f}_{n,m}\ y + \tilde{g}_{n,m}\ xy +  \tilde{k}_{n,m}.
\end{array} \right\} }
\end{align}

\section{\sc Some Auxiliary Results}\label{sec:three}
In this section, we develop some results that are needed in the sequel for investigating the stability of CHFIS in
Section~\ref{sec:four}.

Let $ - \infty < x_0^* <x_1^* < \ldots < x_N^*$  and  $ - \infty < y_0^* <y_1^* < \ldots < y_N^*$ be a bounded set of real numbers
in x and y axis. Denote $I^* = [x_0^*,x_N^*]$,  $J^* = [y_0^*,y_M^*]$,    $\ S^* = I^* \times J^*$
$ I_n^*=[x_{n-1}^*,x_n^*],\ J_m^*=[y_{m-1}^*,y_m^*] \ \mbox{and} \  S_{n,m}^*= I_n^* \times J_m^*$  for
$n=1,\ldots,N \ \mbox{and}\ m=1,\ldots,M.$.
Similar to  $\phi_n\ \mbox{and}\ \psi_m$ in section~\ref{sec:two}, we construct homeomorphisms $\phi_n^* \ \mbox{and} \ \psi_m^*$
where $\phi_n^*: I^* \to  I_n^*$ and $\psi_m^*:J^* \to J_m^* .$
Now, define the  map \vspace{-0.2cm}
$ R_{n,m} : S_{n,m} \longrightarrow S_{n,m}^*\ $ by
\begin{align*}
R_{n,m}(x,y) = \left(x_{n-1}^* + \frac{x_n^* - x_{n-1}^*}{x_n -  x_{n-1}}(x - x_{n-1}),\  y_{m-1}^* +
\frac{y_m^* - y_{m-1}^*}{y_m - y_{m-1}}(y - y_{m-1})\right)
\end{align*}
and the linear piecewise map $R: S \to S^*$ as \vspace{-0.2cm}
\begin{align} \label{eq:Rdefn}
R(x,y) = R_{n,m}(x,y)  \ \ \mbox{for all} \ \  x \in I_n \ \ and \ \   y \in J_m
\end{align}
Similarly, the maps $K_{n,m} : S_{n,m}^* \longrightarrow S_{n,m}$ and $K: S^* \rightarrow S$ are defined as \vspace{-0.2cm}
\begin{align*}
K_{n,m}(x^*,y^*) = \left( x_{n-1} + \frac{x_n - x_{n-1}}{x_n^* -
x_{n-1}^*} \left(x^* - x_{n-1}^* \right),\ y_{m-1} + \frac{y_m - y_{m-1}}{y_m^* -
y_{m-1}^*} \left(y^* - y_{m-1}^* \right) \right)
\end{align*}
 and \vspace{-0.5cm}
\begin{align}\label{eq:Kdefn}
K(x^*,y^*) = K_{n,m}(x^*,y^*)  \ \ \mbox{for all} \ \  x^* \in I_n^* \ \ and \ \   y^* \in J_m^*.
\end{align}
It is easily seen that $K = R^{-1}$. Set $\xi_{n,m} (x,y) = (\phi_n(x), \psi_m(y))$ and $\xi_{n,m}^* (x,y)=(\phi_n^*(x), \psi^*_m(y))$.
We assume the following invariance of ratio condition for any two sets
$\triangle= \{(x_i,y_j,z_{i,j},t_{i,j}) :i=0,1,\ldots,N \ \mbox{and} \  j=0,1,\ldots,M \}\ $ and
$\triangle^*=\{(x_i^*,y_j^*,z_{i,j},t_{i,j}) :i=0,1,\ldots,N \ \mbox{and} \  j=0,1,\ldots,M \}$ of the
generalized interpolation data points :
\begin{flalign}\label{eq:cond1}
\frac{(x_0-x_N)}{(x_0^* - x_N^*)}= \frac{(x_{n-1}-x_n)}{(x^*_{n-1}-x^*_n)} \quad \mbox{and} \quad
\frac{(y_0-y_M)}{(y_0^* - y_M^*)}= \frac{(y_{m-1}-y_m)}{(y^*_{m-1}-y^*_m)}
\end{flalign}
By~\eqref{eq:cond1}, we observe that for $n = 1,2,\ldots,N$ and $m=1,2,\ldots,M$, \vspace{-0.2cm}
\begin{align} \label{eq:inv}
\left.
\begin{array}{rcl}
\xi_{n,m}^*(x^*,y^*)&=&\left(R \circ \xi_{n,m} \circ K \right)(x^*,y^*)  \\
F_{n,m}^*(x^*,y^*,z,t)&=&F_{n,m}\left( K(x^*,y^*),z,t \right)
\end{array} \right\}
\end{align}
Thus, the dynamical systems $\{S;\xi_{n,m}\}$ and $\{S^*;\xi_{n,m}^*\}$ are equivalent.
Using this equivalence of dynamical systems, we first prove the following proposition needed for establishing the
smoothness of CHFIS in Proposition~\ref{prop:st2}:
\begin{proposition}\label{prop:st1}
Let $\triangle= \{(x_i,y_j,z_{i,j},t_{i,j}) :i=0,1,\ldots,N \ \mbox{and} \  j=0,1,\ldots,M \}\ $ and
$\triangle^*=\{(x_i^*,y_j^*,z_{i,j},t_{i,j}) :i=0,1,\ldots,N \ \mbox{and} \  j=0,1,\ldots,M \}$
be any two sets of  generalized interpolation data for the IFS $ \{\Re^4;\om_{n,m}(x,y,z,t), n=1,\ldots,N; m=1,\ldots,M \}$
and  $\{\Re^4;\om_{n,m}^*(x^*,y^*,z,t), n=1,\ldots,N; m=1,\ldots,M \}$  respectively,  with the same choice of free variables and
constrained  variable. Let the points of $\triangle$ and $\triangle^*$ satisfy the invariance of ratio condition~\eqref{eq:cond1}.
Then, $F$, as defined in Section~\ref{sec:two}, is the CHFIS  associated with
$\{\Re^4;\om_{n,m}(x,y,z,t),\ n=1,\ldots,N; m=1,\ldots,M \}$
iff  $\ F\circ R^{-1}$ is the CHFIS  associated with $\{\Re^4;\om_{n,m}^*(x^*,y^*,z,t),\ n=1,\ldots,N; m=1,\ldots,M\}$,
where R is defined by \eqref{eq:Rdefn}.
\end{proposition}

\begin{proof}
Let $F(x,y)$ be the CHFIS associated with the IFS $\{\Re^4;\om_{n,m}(x,y,z,t), n=1,\ldots,N, \\m=1,\ldots,M\}$.
It follows from \eqref{eq:inv} that, \vspace{-0.5cm}
\begin{align*}
F_{n,m}^*\left(\xi_{n,m}^{*-1}(x^*,y^*),\ F \circ R^{-1}(\xi_{n,m}^{*-1}(x^*,y^*))\right)
&= F_{n,m}\left(K \circ \xi_{n,m}^{*-1}(x^*,y^*),\ F \circ R^{-1}(\xi_{n,m}^{*-1}(x^*,y^*))\right)\\
&= F_{n,m}\left(\xi_{n,m}^{-1}\circ K (x^*,y^*), F(\xi_{n,m}^{-1}\circ K (x^*,y^*))\right) \\
&= F_{n,m} \left(\phi_n^{-1}(x),\ \psi_m^{-1}(y),\ F(\phi_n^{-1}(x),\ \psi_m^{-1}(y)) \right) \\
&= F(x,y) \\
&= F \circ R^{-1}(x^*,y^*)
\end{align*}
Thus,  $F \circ R^{-1}$ is the CHFIS associated with $\{\Re^4;\om_{n,m}^*(x^*,y^*,z,t), n = 1,\ldots,N, m =  1,\ldots,M\}$.

Conversely, assume that the above identity holds for CHFIS $F \circ R^{-1}$. \\
Then,  for $x \in I_n$ and $y \in  J_m$,
\begin{align*}
F_{n,m}\left(\xi_{n,m}^{-1}(x,y),\ F(\xi_{n,m}^{-1}(x,y))\right)
&= F_{n,m}^*\left(K^{-1}\circ \xi_{n,m}^{-1}(x,y),\ F(\xi_{n,m}^{-1}(x,y))\right) \\
&= F_{n,m}^* \left(\xi_{n,m}^{*-1}\circ R (x,y),\ F \circ R^{-1}(\xi_{n,m}^{*-1}\circ R (x,y)) \right) \\
&= F_{n,m}^* \left(\phi_n^{*-1}(x^*),\ \psi_m^{*-1}(y^*),\ F \circ R^{-1} (\phi_n^{*-1}(x^*),\ \psi_m^{*-1}(y^*)) \right) \\
&= F \circ R^{-1}(x^*,y^*) \\
&= F (x,y)
\end{align*}
Hence,  $F(x,y)$ is the CHFIS associated with the IFS $\{\Re^4;\om_{n,m}(x,y,z,t), \ n = 1,\ldots,N, \ \mbox{and} \\ m = 1,\ldots,M\}$.
\qed
\end{proof}

Let $X=(x_1,x_2)$ , $Y=(y_1,y_2) \in \mathbb{R}^2$ and $d_M(X,Y)=\sum\limits_{i=1}^2 |x_i-y_i|$ be the \emph{Manhattan metric}
on $\mathbb{R}^2.$ A function $F : \mathbb{R}^2  \rightarrow \mathbb{R}$ is said to be a \emph{Lipschitz function of order~$\delta$}
(written as $F \in Lip\ \delta$) if it satisfies the condition $|F(X)-F(Y)| \leq c \ \left[d_M(X,Y)\right]^\delta, \ \delta \in (0,1]$
and $c$ is a real number. The real number $\delta$ is called the Lipschitz exponent.


Define
\begin{align*}
\begin{matrix}
&Q_n(F_1(X))&:= &\sum\limits_{r_1,\ldots,r_n = 1}^N\ \sum\limits_{s_1,\ldots,s_n = 1}^N
\bigg( \chi_{S_{r_1,\ldots,r_n,s_1,\ldots,s_n}}(X)
\frac{b_{r_1,\ldots,r_n,s_1,\ldots,s_n}}{|S_{r_1,\ldots,r_n,s_1,\ldots,s_n}|} \bigg)&
\end{matrix}
\end{align*}
where, $\ \chi_{S_{r_1,\ldots,r_n,s_1,\ldots,s_n}}(X)=\Bigg\{
\begin{array}{ll}
1, & X \in S_{r_1,\ldots,r_n,s_1,\ldots,s_n} \\
0, & \mbox{otherwise}.
\end{array}$ \vspace{0.3cm}

and $ b_{r_1,\ldots,r_n,s_1,\ldots,s_n} = \int\limits_{I_{r_1,\ldots,r_n}} \int\limits_{J_{s_1,\ldots,s_n}} F_1(x,y)\ dy\ dx$.
\vspace{0.5cm}

It is known \cite{janani08_1} that $Q_n(F_1(X))$ converges to $F_1(X)$ uniformly with respect to Manhattan metric.
Using this fact and finding a bound on maximum distance between $Q_n(F_1(X))$ and $Q_n(F_1(\bar{X}))$,
the Lipschitz exponent of CHFIS $F_1$ is found in the following proposition
when $p_{n,m}$  and $q_{n,m},$ given by \eqref{eq:pnmqnm},  belong to $ Lip\ 1$.

\begin{proposition}\label{prop:st2}
Let $F_1(x,y), \ 0 \leq x ,y \leq \frac{1}{2},$ be the CHFIS for the interpolation data
$\triangle_0= \{(x_i,y_j,z_{i,j},t_{i,j}) :i,j=0,1,\ldots,N \}\ $ where, $x_0=y_0=0 \ \mbox{and}\ x_N=y_N=1/2 .$
Then, $F_1 \in Lip\ \delta $ for some $\de \in (0,1].$
\end{proposition}
\begin{proof}
Using the values of $e_{n,m},\ f_{n,m},\ g_{n,m},\ k_{n,m},\
\tilde{e}_{n,m},\ \tilde{f}_{n,m},\ \tilde{g}_{n,m}\ $  and $\
\tilde{k}_{n,m}$ as defined in~\eqref{eq:coef},  we find that, for
$\ 0 < x < \bar{x} < \frac{1}{2}\ $ and $\ 0 < y < \bar{y} <
\frac{1}{2}\ $,
\begin{align*}
\begin{array}{lll}
&|p_{n,m}(x,y) - p_{n,m}(\bar{x},\bar{y})|  \leq 8 \cdot [(1 +\al)\ Z_{max} + \bt\ T_{max}] \cdot [|x - \bar{x}|+ |y - \bar{y}| ]\\
& |q_{n,m}(x,y) - q_{n,m}(\bar{x},\bar{y})|  \leq 8 \cdot [( 1 +\gm)\ T_{max} ] \cdot [|x - \bar{x}|+ |y - \bar{y}| ]
\end{array}
\end{align*}
where,  $Z_{max}=\max \{|z_{i,j}-z_{k,l}|\ \}$  and $T_{max}= \max
\{|t_{i,j}-t_{k,l}|\ \}\ $ for  $i,j,k,l =0,\ldots,N$. Therefore,
the bound on $|Q_n(F_1(X))- Q_n(F_1(\bar{X}))|$ is obtained as
\begin{align}\label{eq:stmain}
|Q_n&(F_1(X)) - Q_n(F_1(\bar{X}))|  \nno \\ & \leq M \big[d_M(X,\bar{X}) \big] \bigg\{[1 + ( \ \bar{\Om})+\ldots+
(\ \bar{\Om})^{(m-2)} ]  \nno \\ & \quad  \mbox{}+ |\bt| \cdot [( \ \bar{\Gm}) [1 + \ldots+(\ \bar{\Om})^{(m-4)} ] +
(\ \bar{\Gm})^{2}[1 + \ldots+ ( \ \bar{\Om})^{(m-5)} ] + \ldots + (\ \bar{\Gm})^{m-3} ] \bigg\}
\end{align}
The above inequality~\eqref{eq:stmain} is similar to the Inequality $4.1$ of ~\cite{janani08_1} except that,
in this case we have, $\bar{\Om}=\Om=\Th$  and $\bar{\Gm}=\Gm$. By employing the same technique of proof here and using the inequalities
$ 0 < -(|x|+|y|)^{\tau} \log (|x|+|y|) \leq \frac{1}{\tau e} $ and
$ 0 < (|x|+|y|)^{\tau} (\log (|x|+|y|))^2 \leq \frac{4}{\tau^2 e^2}$,
we find that, for different magnitudes of $\bar{\Om}$ and $\bar{\Gm}$
$(\bar{\Om} > 1, \bar{\Om} = 1 \ \mbox{or}\ \bar{\Om} < 1 ; \ \bar{\Gm} > 1, \bar{\Gm} = 1 \ \mbox{or}\ \bar{\Gm} < 1)$,
the bound on $|Q_n(F_1(X))- Q_n(F_1(\bar{X}))|$ is obtained as
$|Q_n(F_1(X)~-~Q_n(F_1(\bar{X}))|\leq~M\big[d_M(X,\bar{X})\big]^{\delta}$ for some $\de \in (0,1].$ Using the fact
$Q_n(F_1(X))$ converges to $F_1(X)$ uniformly with respect to Manhattan metric, we get $F_1 \in Lip\ \delta .$
\qed
\end{proof}
\begin{corollary}\label{cor:st1}
Let $F_1$ be the CHFIS for the generalized interpolation data
$\triangle= \{(x_i,y_j,z_{i,j},t_{i,j}) :i,j=0,1,\ldots,N \}\ $.
Then, there exist constants $\bar{K}$ and $\delta$, independent of
$z_{n,m}$ and $t_{n,m}$, such that
\begin{align}\label{eq:MKbnd}
|Q_n(F_1(X))-Q_n(F_1(\bar{X}))| \leq M  \bar{K}\
\big[d_M(X,\bar{X})\big]^{\delta}
\end{align}
where,  $X=(x,y)\ \mbox{and}\ \bar{X}=( \bar{x},\bar{y})$.
\end{corollary}
\begin{proof}
 Define a linear homeomorphism $R:S \longrightarrow [0,\frac{1}{2}] \times
 [0,\frac{1}{2}]$ that transforms the given interpolation data $\triangle$ to the data $\triangle_0^* =$
$ \{(0,0,z_{0,0},t_{0,0}),(x^*_1,0,z_{1,0},t_{1,0}),
\ldots,(\frac{1}{2},0,z_{N,0},t_{N,0}),\\(0,y^*_1,z_{0,1},t_{0,1}),
\ldots,  (0,\frac{1}{2},z_{0,N},t_{0,N}),\ldots,(\frac{1}{2},\frac{1}{2},z_{N,N},t_{N,N}) \}$.
Proposition (\ref{prop:st2}) applied on the data $\triangle_0^* $
gives $ |Q_n(F_1 \circ R^{-1} (X^*))-Q_n(F_1 \circ R^{-1}(\bar{X^*}))| \leq M
\big[d_M(X^*,\bar{X^*})\big]^{\delta}.$ Now, using Proposition
(\ref{prop:st1}) with $\triangle$ and $\triangle_0^* $, the corollary follows.
\qed
\end{proof}

\section{\sc Stability of CHFIS}\label{sec:four}

To prove the main stability result of CHFIS,  we need to investigate its stability with respect to
perturbations in
independent variables, the dependent variable and the hidden variable of the generalized interpolation data. We develop these
results first.

The following theorem gives the effect on stability of CHFIS due to perturbation in independent variables:

\begin{theorem} \label{th:stth1}
Let  $F_1$ and $G_1,$ with the same choice of free variables and constrained variable, be CHFIS respectively
for the generalized interpolation data $\triangle= \{(x_i,y_j,z_{i,j},t_{i,j}) :i,j=0,1,\ldots,N \} $ and
$\triangle^*=\{(x_i^*,y_j^*,z_{i,j},t_{i,j}) :i,j=0,1,\ldots,N \}$ that
satisfy the the invariance of ratio condition~\eqref{eq:cond1}  and  $S^* \subset S$, where $S^*=[x_0^*,x_N^*] \times [y_0^*,y_M^*]$
and $S=[x_0,x_N] \times [y_0,y_M]$.
Then, \vspace{-0.5cm}
\begin{align}\label{eq:st1}
\lefteqn{\left\|F_1-G_1\right\|_{\infty}}\nno \\ & \quad \leq \bar{M} \left[\frac{2\bt\gm}{(1-\al)(1-\gm)} +
\frac{(1+\al)}{(1-\al)}\right]
\max \bigg\{\left(|x_n-x_n^*| + |y_m-y_m^*|\right)^{\delta} : n,m =0,\ldots,N \bigg \}
\end{align}
\end{theorem}
\begin{proof}
Let the function $G(x,y)$ corresponding to the generalized interpolation points
$ \triangle^*$  be defined as
$ G(x,y) = G_{n,m}^*\left(\phi_n^{*-1}(x),\psi_m^{*-1}(y),G(\phi_n^{*-1}(x),\psi_m^{*-1}(y))\right)$ for
$x \in [x_{n-1}^*,x_n^*]$ and $y \in [y_{m-1}^*,y_m^*]$.
By~\eqref{eq:cond1} and~\eqref{eq:inv}, we observe,
\begin{align*}
G_1\left(R(\phi_n(x),\psi_m(y))\right)&\mbox{}= F_{n,m}^1
\left(x,y,G \circ R (x,y) \right)
\\ & \mbox{}= \al_{n,m}\ (G_1 \circ R)(x,y) + \bt_{n,m}\ (G_2 \circ R)(x,y) \\ & \quad \mbox{}+
e_{n,m}\ x +f_{n,m}\ y +g_{n,m}\ xy + k_{n,m}\\
G_2\left(R(\phi_n(x),\psi_m(y)\right) &\mbox{}=F_{n,m} \left(x,y,G
\circ R (x,y) \right)
\\ &\mbox{}= \gm_{n,m}\ (G_2 \circ R)(x,y) + \tilde{e}_{n,m}\ x +\tilde{f}_{n,m}\ y +
\tilde{g}_{n,m}\ xy + \tilde{k}_{n,m}
\end{align*}
Thus, we have, \vspace{-0.5cm}
\begin{align*}
|G_1(R(\phi_n(x),& \psi_m(y)))-F_1(\phi_n(x),\psi_m(y))| \\
& \leq |\al_{n,m}|\  |G_1 \circ R(x,y)- F_1(x,y)| + |\bt_{n,m}| |G_2 \circ R(x,y)- F_2(x,y)| \\
& \leq |\al|\  |G_1 \circ R(x,y)- F_1\circ R(x,y)| + |\al|\  |F_1 \circ R(x,y)- F_1(x,y)| \\
& \quad \mbox{} + |\bt|\  |G_2 \circ R(x,y)- F_2(x,y)|
\end{align*}
Using Proposition~\ref{prop:st2} and Corollary~\ref{cor:st1}, it follows that there exist constants $\bar{M}$ and $\delta$
which are independent of $z_{n,m}$  and $t_{n,m}$  such that  $|F_1(X)-F_1(\bar{X})| \leq \bar{M}\  d_M(X,\bar{X})^{\delta}$
where,  $X=(x,y)\ \mbox{and}\ \bar{X}=( \bar{x},\bar{y})$. Since $(x,y) \in S $ implies $ (x,y) \in S_{n,m} $ for some
$n,m = 1,\ldots,N$, it is easily seen that $R(x,y) \in S_{n,m}^* $ which in turn gives
$d_M(X,R(X))\leq \max \{(|x_n-x_n^*| + |y_m-y_m^*|):n,m = 0,1,\ldots,N \}.$
Thus, the last inequality reduces to \vspace{-0.5cm}
\begin{align*}
\|G_1 \circ R - F_1 \|_{\infty} & \leq |\al|\  \|G_1-F_1\|_{\infty}+ |\bt|\ \|G_2 \circ R - F_2\|_{\infty} \nno \\
& \quad \mbox{}+ |\al|\  \bar{M} \max \left\{ \left(|x_n-x_n^*| + |y_m-y_m^*|\right)^{\delta}: n,m=0,1,\ldots,N \right\}
\end{align*}
Now, $\|G_1  - F_1 \|_{\infty}  \leq \|G_1  - F_1 \circ R^{-1}\|_{\infty}+ \|F_1 \circ R^{-1} - F_1 \|_{\infty}. $
Therefore,\vspace{-0.5cm}
\begin{align}\label{eq:th1tempbnd}
\|G_1  - F_1 \|_{\infty}
& \leq \frac{\bt}{(1-\al)}\ \|G_2 \circ R - F_2\|_{\infty}   \nno \\
& \quad \mbox{}+ \frac{(1+\al)}{(1-\al)}\ \bar{M}
\max \left \{\left(|x_n-x_n^*| + |y_m-y_m^*|\right)^{\delta}:n,m =0,1,\ldots,N  \right \}
\end{align}
In order to find an upper bound of $\|G_2 \circ R - F_2\|_{\infty}$ in~\eqref{eq:th1tempbnd} , we observe that, \vspace{-0.5cm}
\begin{align*}
|G_2 \circ R (\phi_n(x),\psi_m(y)) &- F_2(\phi_n(x),\psi_m(y))| \\ &
\leq |\gm|\  \bigg\{|G_2 \circ R(x,y)-F_2 \circ R(x,y)| + |F_2 \circ R(x,y) - F_2(x,y)|\bigg\}
\end{align*}
The above inequality holds for all $n,m = 1,2 \ldots N$. So, \vspace{-0.5cm}
\begin{align*}
\|G_2 \circ R - F_2\|_{\infty}  \leq |\gm|\  \|G_2-F_2\|_{\infty}
+ |\gm|\  \bar{M} \max \left \{ \left(|x_n-x_n^*| + |y_m-y_m^*|\right)^{\delta}: n,m = 0,\ldots,N \right \}
\end{align*}
Hence, \vspace{-0.5cm}
\begin{align*}
\|G_2 - F_2\|_{\infty}& \leq \bar{M} \frac{(1+\gm)}{(1-\gm)}\
\max \left \{\left(|x_n-x_n^*| + |y_m-y_m^*|\right)^{\delta}:n,m =0,\ldots,N \right \} \nno
\end{align*}
Therefore, \vspace{-0.5cm}
\begin{align} \label{eq:bnd2}
\|G_2 \circ R - F_2\|_{\infty}  \leq  \bar{M} \frac{(2 \gm)}{(1-\gm)}\
\max \left \{ \left(|x_n-x_n^*| + |y_m-y_m^*|\right)^{\delta}:n,m = 0,1,\ldots,N \right \}
\end{align}
Substituting  inequality~\eqref{eq:bnd2} in  inequality~\eqref{eq:th1tempbnd}, we get the required bounds.   \qed
\end{proof}

The following theorem gives stability of CHFIS when there is a perturbation in the dependent variable.

\begin{theorem}\label{th:stth2}
Let $F_1$ and $G_1,$ with the same choice of free variables and constrained variable, be CHFIS respectively
for the generalized interpolation data $\triangle= \{(x_i,y_j,z_{i,j},t_{i,j}) :i,j=0,1,\ldots,N \} $ and
$\triangle^*=\{(x_i,y_j,z_{i,j}^*,t_{i,j}) :i,j=0,1,\ldots,N \}$. Then,
\begin{align}\label{eq:st2}
\left\|F_1-G_1\right\|_{\infty} \leq \frac{4 (1+\al)}{(1-\al)}\ \max \bigg\{|z_{n,m}-z_{n,m}^*|: n,m = 0,1,\ldots,N \bigg\}
\end{align}
\end{theorem}
\begin{proof}
In view of Proposition~\ref{prop:st1}, we may assume $[x_0,x_N]\times [y_0,y_N] = [0,\frac{1}{2}] \times [0,\frac{1}{2}]$.
Since the independent variables and hidden variable are same in both the interpolation data, the self affine FIS are the same
i.e $F_2(x,y)=G_2(x,y)$. The value of $e_{n,m},\  f_{n,m},\  g_{n,m},\  k_{n,m}$ differs from
$e_{n,m}^*,\  f_{n,m}^*,\  g_{n,m}^*,\  k_{n,m}^*$ as the perturbation  occurs in dependent variable.

Therefore, for $n,m = 1,2,\ldots,N,$ \vspace{-0.5cm}
\begin{align*}
|F_1&(\phi_n(x) ,\psi_m(y))  - G_1(\phi_n(x),\psi_m(y))| \\
& \leq |\al|\  |F_1(x,y)-G_1(x,y)|\\ &  \mbox{}+ (1 + \al)\left[2|x-\frac{1}{2}|+ 2|y-\frac{1}{2}|+4|xy-\frac{1}{4}| + 1 \right]
\max \bigg\{|z_{n,m}-z_{n,m}^*|: n,m = 0,1,\ldots,N \bigg\}
\end{align*}
Since $\left[2|x-\frac{1}{2}|+ 2|y-\frac{1}{2}|+4|xy-\frac{1}{4}| + 1 \right] \leq 4$, the required bounds for stability is obtained
from the above inequality.\qed
\end{proof}

When the hidden variable is perturbed, both the self affine function $F_2$ and CHFIS $F_1$  gets perturbed.
Proposition~\ref{prop:st3} describes the stability of self affine function $F_2$ when the hidden variable is perturbed.
Using this proposition, the stability of CHFIS is described in Theorem \eqref{th:stth3} when the hidden variable is perturbed.

\begin{proposition}\label{prop:st3}
Let $F_1$ and $G_1,$ with the same choice of free variables and constrained variable, be CHFIS respectively
for the generalized interpolation data $\triangle= \{(x_i,y_j,z_{i,j},t_{i,j}) :i,j=0,1,\ldots,N \} $ and
$\triangle^*=\{(x_i,y_j,z_{i,j},t_{i,j}^*) :i,j=0,1,\ldots,N \}$. Then,
\begin{align}\label{eq:bndF2}
\left\| F_2-G_2 \right\|_{\infty} \leq \frac{4 (1+\gm)}{(1-\gm)} \max \{|t_{n,m}-t_{n,m}^*| : n = 0,1,\ldots,N; m = 0,1,\ldots,N \}
\end{align}
\end{proposition}
\begin{proof}
By Proposition~\ref{prop:st1}, we may assume $[x_0,x_N] \times [y_0,y_N] = [0,\frac{1}{2}] \times [0,\frac{1}{2}]$.
Hence, \vspace{-0.3cm}
\begin{align*}
|F_2(\phi_n(x),\psi_m(y))& - G_2(\phi_n(x),\psi_m(y))| \\
& \leq |\gm|\ |F_2(x,y)-G_2(x,y)|  + \left[2|x-\frac{1}{2}|+ 2|y-\frac{1}{2}|+4|xy-\frac{1}{4}| + 1 \right] \times \\
& \quad \mbox{} \times (1+\gm) \left\{ \max \{|t_{n,m}-t_{n,m}^*| : n,m = 0,1,\ldots,N \} \right\}
\end{align*}
The above inequality is true for all $\phi_n(x)$ and $\psi_m(y)$;
$n,m = 1,2,\ldots,N$  giving the required bounds for  $\left\|
F_2-G_2 \right\|_{\infty}$. \qed
\end{proof}

\begin{theorem}\label{th:stth3}
Let $F_1$ and $G_1,$ with the same choice of free variables and constrained variable, be CHFIS respectively
for the generalized interpolation data $\triangle= \{(x_i,y_j,z_{i,j},t_{i,j}) :i,j=0,1,\ldots,N \} $ and
$\triangle^*=\{(x_i,y_j,z_{i,j},t_{i,j}^*) :i,j=0,1,\ldots,N \}$. Then,
\begin{align}
\left\|F_1-G_1\right\|_{\infty} \leq \frac{8\ \bt}{(1-\al)(1-\gm)} \max \bigg\{|t_{n,m}-t_{n,m}^*|: n,m = 0,1,\ldots,N \bigg\}
\end{align}
\end{theorem}
\begin{proof}
We may assume $[x_0,x_N]\times [y_0,y_N] = [0,\frac{1}{2}] \times [0,\frac{1}{2}]$ by Proposition~\ref{prop:st1}.
Thus, \vspace{-0.3cm}
\begin{align*}
|F_1&(\phi_n(x),\psi_m(y))  - G_1(\phi_n(x),\psi_m(y))| \\
& \leq  |\al| |F_1(x,y)-G_1(x,y)| + |\bt| |F_2(x,y)-G_2(x,y)| \\
& \quad \mbox{} +  |\bt| \left[2|x-\frac{1}{2}|+ 2|y-\frac{1}{2}|+4|xy-\frac{1}{4}| + 1 \right]
\max \{|t_{n,m}-t_{n,m}^*|: n,m = 0,1,\ldots,N \}
\end{align*}

Using equation(\ref{eq:bndF2}) and since $\left[2|x-\frac{1}{2}|+ 2|y-\frac{1}{2}|+4|xy-\frac{1}{4}| + 1 \right] \leq 4$,
we get the desired bounds in \eqref{th:stth3}. \qed
\end{proof}
Theorems (\ref{th:stth1})-(\ref{th:stth3}) suggest the following definition of a metric on a generalized interpolation data
needed to formulate our main result on stability for CHFIS.
\begin{definition}
Let $S(\triangle)$ be the set of generalized interpolation data. The metric
$d(\triangle^1, \triangle^2)$  on the set $S(\triangle) \subset \Re^4$ is defined as: \vspace{-0.2cm}
\begin{align}\label{eq:metric}
d(\triangle^1, \triangle^2)  & = \frac{8\ \bt}{(1-\al)\ (1-\gm)} \max \{|t_{n,m}^1-t_{n,m}^2|: n,m = 0,1,\ldots,N \} \nno \\
& \quad \mbox{} + \frac{4 (1+\al)}{(1-\al)} \max \{|z_{n,m}^1-z_{n,m}^2|: n,m = 0,1,\ldots,N \} \nno \\
& \quad \mbox{} + \left[\frac{2\ \bt\ \gm}{(1-\al)(1-\gm)} + \frac{(1+\al)}{(1-\al)}\right] \bar{M} \times \nno \\
& \quad \mbox{} \times \max \{\left(|x_n^1-x_n^2| + |y_m^1-y_m^2|\right)^{\delta} : n,m = 0,1,\ldots,N\}
\end{align}
where $\triangle^m = \{(x_i^m,y_j^m,z_{i,j}^m,t_{i,j}^m),:i,j=0,1,\ldots,N  \} \in S(\triangle),
\ m=1,2$  and  satisfy invariance of ratio condition~\eqref{eq:cond1}.
\end{definition}
Using Theorems \eqref{th:stth1}-\eqref{th:stth3}, the main stability result for CHFIS is now obtained as follows:
\begin{theorem}\label{th:stmain}
Let $F_1$ and $G_1,$ with the same choice of free variables and
constrained variable, be CHFIS respectively for the generalized
interpolation data $\triangle= \{(x_i,y_j,z_{i,j},t_{i,j})
:i,j=0,1,\ldots,N \}\ $ and
$\triangle^*=\{(x_i^*,y_j^*,z_{i,j}^*,t_{i,j}^*) :i,j=0,1,\ldots,N
\}$ that satisfy the invariance of ratio condition~\eqref{eq:cond1}.
Then,
\begin{align*}
\|F_1-G_1\|_{\infty} \leq d\left(\triangle,\triangle^*\right).
\end{align*}
\end{theorem}
\begin{proof}
 Let $\triangle^*_1 = \{(x_i,y_j,z_{i,j},t_{i,j}^*) :i,j=0,1,\ldots,N  \} $ and
 $\triangle^*_2 = \{(x_i,y_j,z_{i,j}^*,t_{i,j}^*) :i,j=0,1,\ldots,N \}$
 be two generalized interpolation data having CHFIS $F_1^*$ and $F_1^{**}$ respectively with the
same choice of free variables and constrained variable as for $F_1$ and $G_1$. By considering the pairs
$\left(\triangle,\triangle^*_1\right), \ \left(\triangle^*_1,\triangle^*_2\right)$  and $\left(\triangle^*_2,\triangle^*\right)$
and applying Theorems (\ref{th:stth1})-(\ref{th:stth3}) for these sets of data with appropriate CHFIS,
it follows that,\vspace{-0.5cm}
\begin{align*}
\|F_1-G_1\|_{\infty} \leq \|F_1-F_1^{*}\|_{\infty} + \|F_1^{*}-F_1^{**}\|_{\infty} + \|F_1^{**}-G_1\|_{\infty}
\leq d\left(\triangle,\triangle^* \right). \hspace{2cm} \qed
\end{align*}
\end{proof}

\section{Error bounds for a Sample Surface}\label{sec:five}
Consider the sample CHFIS (\textit{c.f. Fig.~\ref{fig:chfisorig}}) generated by the data in rows $1-3$ of Table~\ref{tab:sttable1}
with $\al_{n,m}=0.7$,  $\bt_{n,m}=0.4$ and $\gm_{n,m}=0.5$. Perturbed values in independent variables $x_n$ and $y_m$ are chosen
such that the invariance of ratio condition~\eqref{eq:cond1} is satisfied.
 However, perturbed values of dependent variable $z_{n,m}$ and hidden variable $t_{n,m}$ are randomly generated.

Figs.~\ref{fig:chfisxy1} and~\ref{fig:chfisxy2} are simulations of CHFIS generated corresponding to perturbations
in independent variables $x_n$ and $y_m$ \big(\textit{c.f. Table~\ref{tab:sttable1} Case~(i(a)) and Case~(i(b))}\big).
The bound on error $\|F - G\|_{\infty}$ (\textit{c.f. Theorem~\ref{th:stth1}}) for these simulated surfaces are found
to be $0.0217$  and $2.1667$ when $\max \big\{\left(|x_n-x_n^*| + |y_m-y_m^*|\right) : n,m =0,1,2 \big \}$
is $0.002$ and $0.2$ respectively (\textit{c.f. Table~\ref{tab:sttable2}}) .

Figs.~\ref{fig:chfisz1} and~\ref{fig:chfisz2} give simulations of CHFIS obtained by perturbing the values of dependent variable
$z_{n,m}$ \big(\textit{c.f. Table~\ref{tab:sttable1} Case~(ii(a)) and Case~(ii(b))}\big). The bound on error $\|F - G\|_{\infty}$
(\textit{c.f. Theorem~\ref{th:stth2}}) for these simulations equals $0.0227$ and $2.2667$
when $\max \{|z_{n,m}^1-z_{n,m}^2|: n,m = 0,1,2 \}$ is $0.001$ and $0.1$ respectively
(\textit{c.f. Table~\ref{tab:sttable2}}).

Figs.~\ref{fig:chfist1} and~\ref{fig:chfist2} demonstrate the effect of perturbations in hidden variable $t_{n,m}$
\big(\textit{c.f. Table~\ref{tab:sttable1} Case~(iii(a)) and Case~(iii(b))}\big). The bound on error $\|F - G\|_{\infty}$
(\textit{c.f. Theorem~\ref{th:stth3}}) for these simulated surfaces  equals $0.0213$ and $2.1333$
when $\max \{|t_{n,m}^1-t_{n,m}^2|: n,m = 0,1,2 \}$ is $0.001$ and $0.1$ respectively
(\textit{c.f. Table~\ref{tab:sttable2}}).

Finally,  Fig.~\ref{fig:chfistr1} and~\ref{fig:chfistr2} give the perturbed images that is simulated by simultaneously  using
perturbed independent variable, perturbed dependent variable $z_{n,m}^*$ and perturbed hidden variable
\big(\textit{c.f. Table~\ref{tab:sttable1} Case~(i(a)),(ii(a))(iii(a)) and Case~(i(b)),(ii(b))(iii(b)) respectively }\big).
The computed bound on error $\|F - G\|_{\infty}$
(\textit{c.f. Theorem~\ref{th:stmain}}) for these perturbations in all the variables is found to be $0.0657$ and $6.5667$
when $\max \big\{\left(|x_n-x_n^*| + |y_m-y_m^*|\right) : n,m =0,1,2 \big \}$ is $0.002$ and $0.2,$
$\max \{|z_{n,m}^1-z_{n,m}^2|: n,m = 0,1,2 \} $ is $0.001$ and $0.1$
and $\max \{|t_{n,m}^1-t_{n,m}^2|: n,m = 0,1,2 \} $ is $0.001$ and $0.1$ respectively.


\renewcommand\arraystretch{2.3}
\begin{table}[!hbp]
\bc \caption{Data points and their perturbations for a sample surface} \vspace{0.2cm}
{ \tiny
\begin{tabular}{|c|c|c|c|c|c|c|c|c|c|c|}\hline
\multirow{2}{15 pt}{\begin{sideways}\textbf{\ Data Points \ }\end{sideways} }
&$(x_n,y_m)$ & (0,0) & (0,1) & (0,2) & (1,0) & (1,1) & (1,2) & (2,0) & (2,1) & (2,2) \cr  \cline{2-11}
&$z_{n,m}$   &  0.3  & 0.5      & 0.6     & 0.7      & 0.4         & 0.6        & 0.8     & 0.5        & 0.6   \cr  \cline{2-11}
&$t_{n,m}$   &  0.3  & 0.4      & 0.5     & 0.7      & 0.8         & 0.5        & 0.6     & 0.8        & 0.9   \cr  \cline{1-11}
\multirow{1}{15 pt}{\textbf{\ Case (i) }}
&(a): $(x_n^*,y_m^*)$ & (0.001,0.001) & (0.001,1) & (0.01,1.999) & (1,0.001) & (1,1) & (1,1.999)
& (1.999,0.001) & (1.999,1) & (1.999,1.999) \cr  \cline{2-11}
&(b): $(x_n^*,y_m^*)$ & (0.1,0.1) & (0.1,1) & (0.1,1.9) & (1,0.1) & (1,1) & (1,1.9)
& (1.9,0.1) & (1.9,1) & (1.9,1.9) \cr  \cline{1-11}
\multirow{1}{15 pt}{\textbf{\ Case (ii) }}
&(a):  $z_{n,m}^*$    & 0.301 & 0.501 & 0.601 & 0.699 & 0.401 & 0.599 & 0.801 & 0.501 & 0.601  \cr  \cline{2-11}
&(b):  $z_{n,m}^*$    & 0.4 & 0.4 & 0.7 & 0.6 & 0.3 & 0.7 & 0.9 & 0.4 & 0.5  \cr  \cline{1-11}
\multirow{1}{15 pt}{\textbf{\ Case (iii) }}
&(a):  $t_{n,m}^*$    & 0.299 & 0.401 & 0.499 & 0.701 & 0.801 & 0.501 & 0.601 & 0.801 & 0.9  \cr  \cline{2-11}
&(b):  $t_{n,m}^*$    & 0.4 & 0.5 & 0.4 & 0.6 & 0.9 & 0.4 & 0.5 & 0.9 & 0.8   \\ \hline
\end{tabular}}
\label{tab:sttable1}
\ec
\end{table}

\renewcommand\arraystretch{2}
\begin{table}[!hbp]
\bc \caption{Error due to perturbation}
\vspace{0.2cm}{\tiny
\begin{tabular}{|c|c|c|c|}\hline
\multicolumn{2}{|c|}{Perturbation}& Maximum Manhattan Metric & Error \\ \hline
\multirow{1}{30 pt}{\textbf{\ \ Case(i) }}
&(a): Perturbed $(x_n^*,y_m^*)$  & 0.002 &0.0217  \cr  \cline{2-4}
&(b): Perturbed $(x_n^*,y_m^*)$  & 0.2   &2.1667  \cr  \cline{1-4}
\multirow{1}{30 pt}{\textbf{\ Case(ii) }}
&(a): Perturbed $z_{n,m}^*$     & 0.001  &0.0227  \cr  \cline{2-4}
&(b): Perturbed $z_{n,m}^*$     & 0.1    &2.2667  \cr  \cline{1-4}
\multirow{1}{30 pt}{\textbf{\ Case(iii) }}
&(a): Perturbed $t_{n,m}^*$     & 0.001  &0.0213  \cr  \cline{2-4}
&(b): Perturbed $t_{n,m}^*$     & 0.1    &2.1333   \\ \hline
\end{tabular}}
\label{tab:sttable2}
\ec
\end{table}

\begin{figure}[!hbp]
\centering  {
\includegraphics[width=4cm]{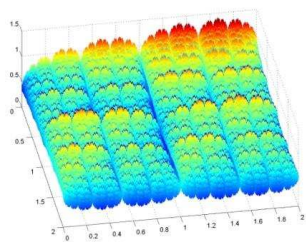}
}\caption{Original Surface \label{fig:chfisorig}}
\end{figure}

\begin{figure}[!hbp]\hspace{2.5cm}
\subfigure[{\scriptsize  Simulated Surface due to perturbation in
independent variables (\textit{c.f. Case~(i(a))})}] {
\includegraphics[width=4cm]{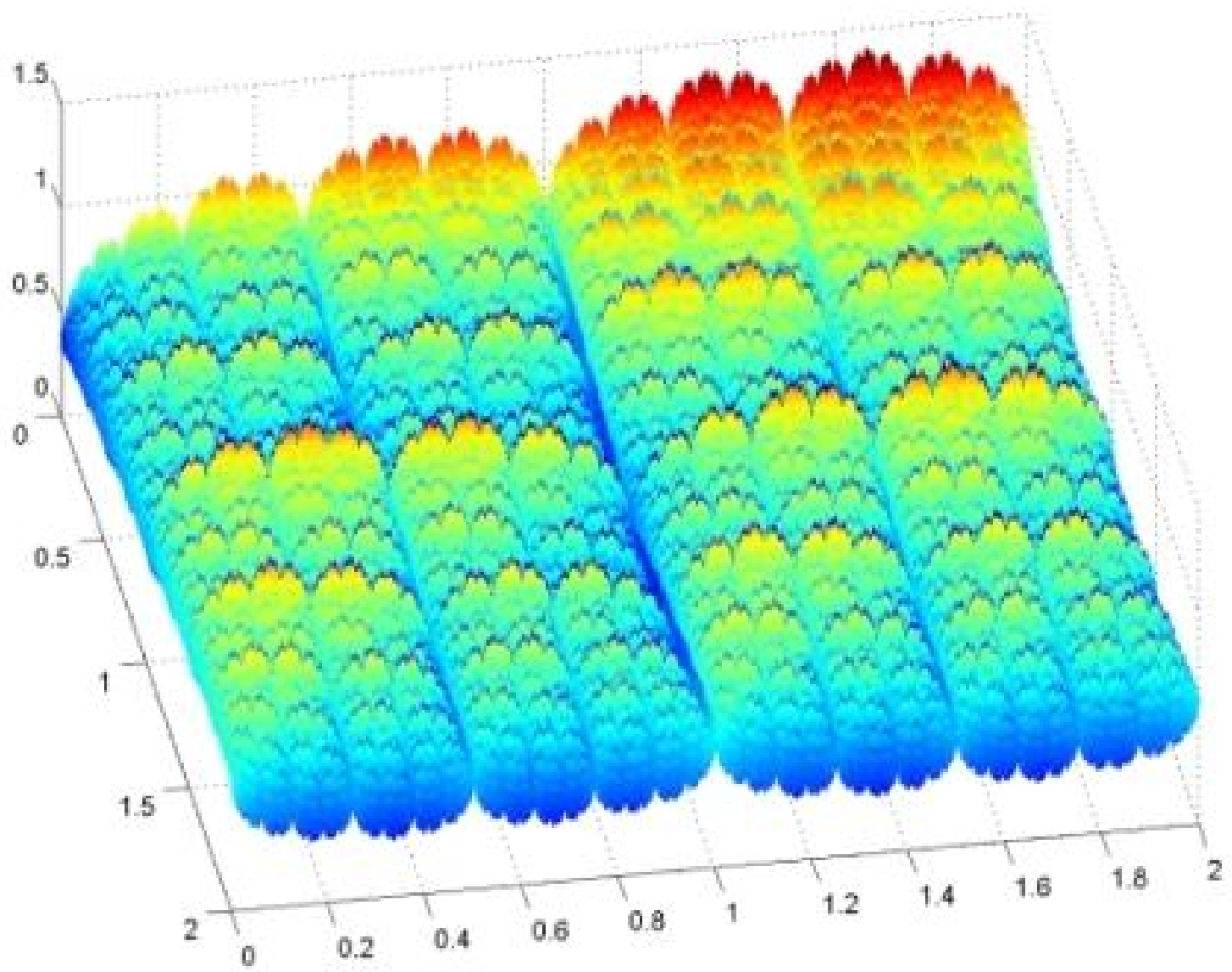}
\label{fig:chfisxy1} } \hspace{3cm} \subfigure [{\scriptsize
Simulated Surface due to perturbation in independent variables
(\textit{c.f. Case~(i(b))})}]
{\includegraphics[width=4cm]{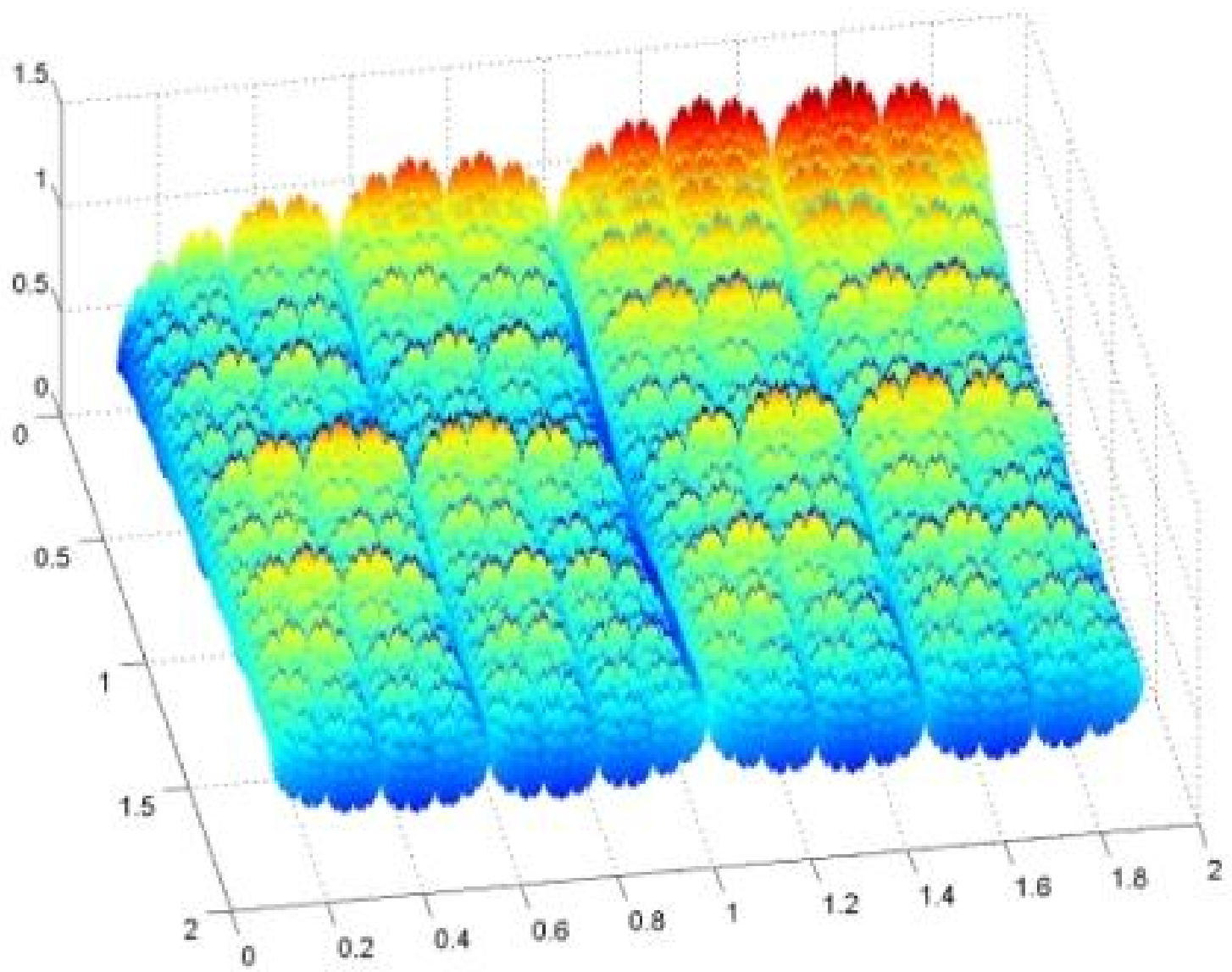} \label{fig:chfisxy2}}
\caption{Simulated Surfaces due to perturbation in independent
variables (\textit{c.f. Case~(i)}) }
\end{figure}

\begin{figure}[!hbp]\hspace{2.5cm}
\subfigure[{\scriptsize Simulated Surface due to perturbation in
dependent variable (\textit{c.f. Case~(ii(a))})}] {
\includegraphics[width=4cm]{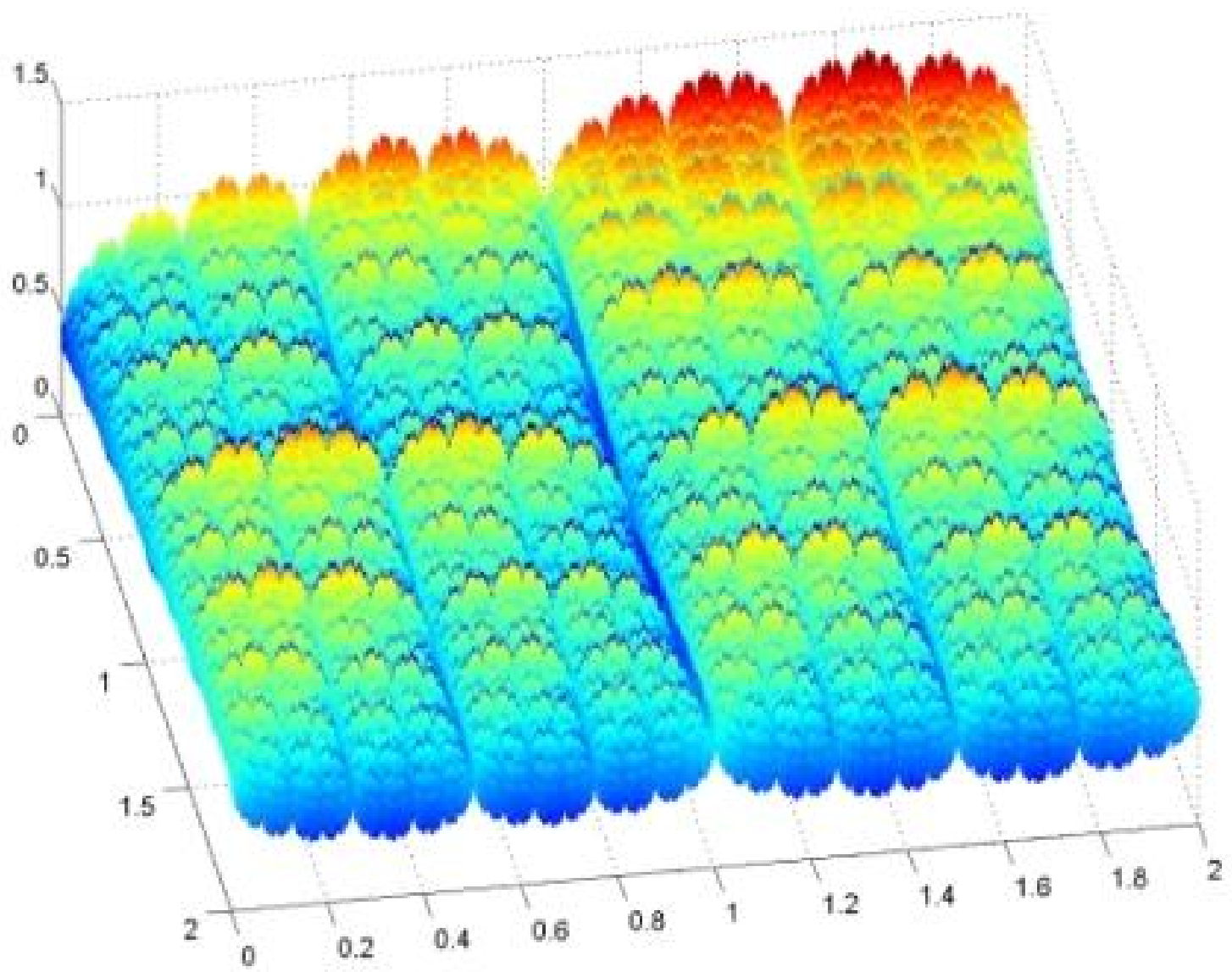}
\label{fig:chfisz1} } \hspace{3cm} \subfigure [{\scriptsize
Simulated Surface due to perturbation in dependent variable
(\textit{c.f. Case~(ii(b))})}] {
\includegraphics[width=4cm]{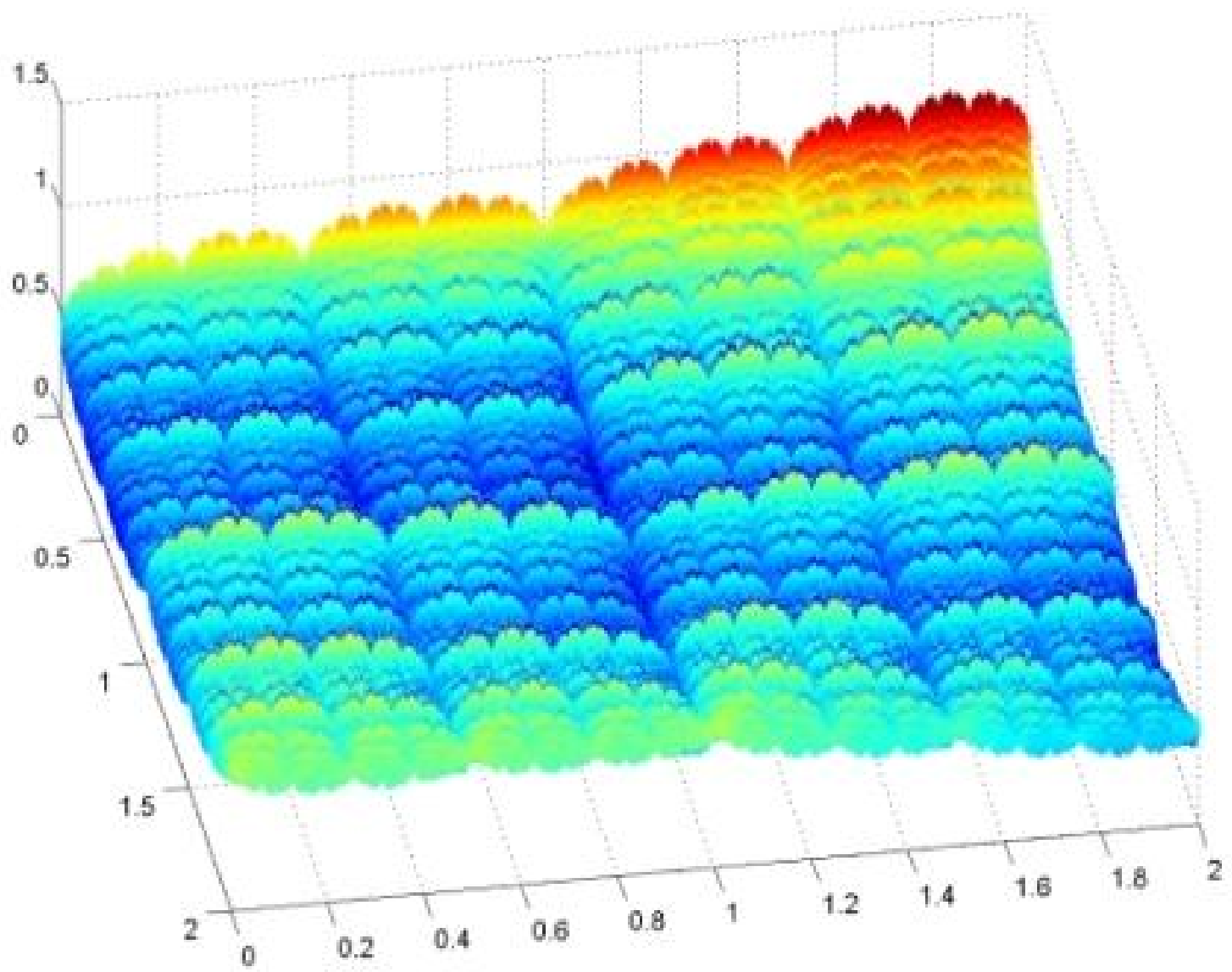}
\label{fig:chfisz2} } \caption{Simulated Surfaces due to
perturbation in dependent variable (\textit{c.f. Case~(ii)})}
\end{figure}

\begin{figure}[!hbp]\hspace{2.5cm}
\subfigure [{\scriptsize Simulated Surface due to perturbation in
hidden variable (\textit{c.f. Case~(iii(a))})}] {
\includegraphics[width=4cm]{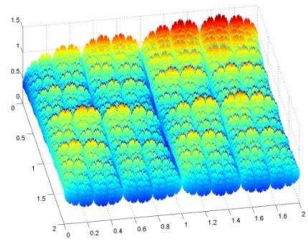}
\label{fig:chfist1} } \hspace{3cm} \subfigure [{\scriptsize
Simulated Surface due to perturbation in hidden variable
(\textit{c.f. Case~(iii(b))})}] {
\includegraphics[width=4cm]{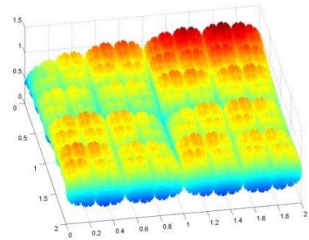}
\label{fig:chfist2}} \caption{Simulated Surfaces due to perturbation
in hidden variable (\textit{c.f. Case~(iii)})}
\end{figure}

\newpage

\begin{figure}[!hbp]\hspace{2.5cm}
\subfigure [{\scriptsize Simulated Surface due to perturbation in
all the variables (\textit{c.f. Case~(i(a)),(ii(a))(iii(a))})}] {
\includegraphics[width=4cm]{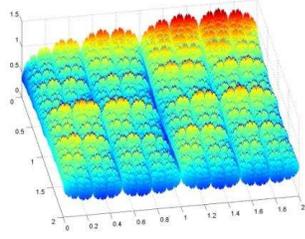}
\label{fig:chfistr1} } \hspace{3cm} \subfigure [{\scriptsize
Simulated Surface due to perturbation in all the variables
(\textit{c.f. Case~(i(b)),(ii(b))(iii(b))})}] {
\includegraphics[width=4cm]{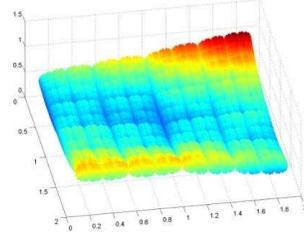}
\label{fig:chfistr2}} \caption{Simulated Surfaces due to
perturbation in all the variables}
\end{figure}

\section{Conclusion}
The present paper explores the stability of CHFIS when there is a
perturbation in independent variables, the dependent variable and
the hidden variable. The stability during the perturbations in all
the variables simultaneously is observed to be the combined
individual effect of perturbations in each variable on the stability
of CHFIS. The bound on error in the approximation of the data
generating function by CHFIS is described individually for each case
of perturbation in independent, dependent and hidden variables.
These bounds together are employed to find the total error bound on
CHFIS when there is perturbation in all the  variables
simultaneously. The stability results found here are illustrated
through a  sample surface. Our results are likely to find
applications in investigations concerning texture of surfaces of
naturally occurring objects like surfaces of rocks~\cite{xie97}, sea
surfaces~\cite{osborne86}, clouds~\cite{zimmermann92} etc.

{\textbf{ Acknowledgement:}}
The author Srijanani thanks CSIR for research grant (No:9/92(417)/2005-EMR-I) for the present work.

\end{document}